\documentclass[12pt]{article}

\usepackage{amssymb}
\usepackage{graphics}

\newcommand{\qed}{$\;\;\;\Box$}
\newenvironment{proof}{\par\smallbreak{\sl Proof.~}}
{\unskip\nobreak\hfill \qed \par\medbreak}

\newcounter{claim}
\renewcommand{\theclaim}{\arabic{claim}}
{\par\medskip\par}

{\qed\par\smallbreak}
\newcommand{\hide}[1]{}

\newcommand{\N}{{\mathbb N}}

\newcommand{\R}{{\mathbb R}}

\newcommand{\Z}{{\mathbb Z}}

\newcommand{\LL}{{\cal L}}

\newcommand{\beq}{\begin{equation}}
\newcommand{\ee}{\end{equation}}

\renewcommand{\d}{\partial}

\newtheorem{thm}{Theorem}[section]
\newtheorem{lemma}[thm]{Lemma}

\newtheorem{defn}[thm]{Definition}

\newtheorem{rem}[thm]{Remark}

\newcommand{\al}{\alpha}
\newcommand{\be}{\beta}
\newcommand{\ga}{\gamma}

\newcommand{\eps}{\varepsilon}
\newcommand{\vphi}{\varphi}

\newcommand{\om}{\omega}

\newcommand{\reff}[1]{(\ref{#1})}

\newcommand{\diag}{\mathop{\rm diag}\nolimits}

\renewcommand{\Im}{\mathop{\mathrm{Im}}\nolimits}

\title{
Robustness of Exponential Dichotomies of Boundary-Value Problems
for General First-Order Hyperbolic Systems
}

\newcounter{thesame}
\author{
I.~Kmit
\ \ \ L.~Recke\ \ \ V.~Tkachenko\\
{\small
Institute of Mathematics, Humboldt University of Berlin,}
\\
{\small Rudower Chaussee 25,
\small D-12489 Berlin, Germany and
}
\\
{\small
Institute for Applied Problems of Mechanics and Mathematics,
}
\\
{\small
Ukrainian Academy of Sciences,
Naukova St.\ 3b,
79060 Lviv, Ukraine
}
\\
{\small   E-mail:
{\tt kmit@informatik.hu-berlin.de}}\\[5mm]
{\small
Institute of Mathematics, Humboldt University of Berlin,}\\
{\small
Rudower Chaussee 25,
D-12489 Berlin, Germany}\\
{\small   E-mail:
{\tt recke@mathematik.hu-berlin.de}}
\\[5mm]
{\small
Institute of Mathematics, Ukrainian Academy of Sciences,}\\
{\small
Tereschenkivska st. 3,
Kiev-4, 01601 Ukraine
}\\
{\small   E-mail:
{\tt vitk@imath.kiev.ua}}
}
\date{}
\begin{document}

\maketitle

\begin{abstract}
\noindent
We examine robustness of  exponential dichotomies  of boundary value problems
for general linear first-order one-dimensional hyperbolic systems. The boundary conditions are
supposed to be of  types ensuring smoothing solutions in finite time, which includes
reflection boundary conditions. We show
that the dichotomy survives in the space of continuous functions under small perturbations
of all coefficients in the differential equations.
\end{abstract}



\section{Introduction and main results}\label{sec:intr}
\renewcommand{\theequation}{{\thesection}.\arabic{equation}}
\setcounter{equation}{0}

The concept of exponential dichotomy plays a crucial role in various aspects of the perturbation
and the stability theory \cite{Coppel,DK,SS,Sam,T}. An important problem here is
robustness of the exponential dichotomy of a system, i.e., its stability  with respect to small
perturbations in the system. This problem is extensively studied in the literature,
e.g., in \cite{JS,NP,Pal,Yi} for finite-dimensional case and  in \cite{BV,ChL1,PS}
for infinite-dimensional case. It should be noted that the hyperbolic case (see, e.g., \cite{RB})
seems more complicated here in comparison to ODEs and parabolic PDEs, mostly due
to worse regularity properties of hyperbolic operators.

\bigskip

We address the issue of stability
 of exponential dichotomies for general linear one-dimensional
first-order hyperbolic systems
\begin{equation}\label{eq:1}
(\partial_t  + a(x,t,\eps)\partial_x + b(x,t,\eps))
 u = 0, \quad x\in(0,1)
\end{equation}
subjected to (nonlocal) boundary conditions
\beq\label{eq:3}
\begin{array}{l}
\displaystyle
u_j(0,t) = \sum\limits_{k=m+1}^np_{jk}(t)u_k(0,t)+\sum\limits_{k=1}^mp_{jk}(t)u_k(1,t),\quad  1\le j\le m,
\nonumber\\
\displaystyle
u_j(1,t) = \sum\limits_{k=1}^mq_{jk}(t)u_k(0,t)+ \sum\limits_{k=m+1}^nq_{jk}(t)u_k(1,t), \quad   m< j\le n.
\nonumber
\end{array}
\ee
Here  $u=(u_1,\ldots,u_n)$ is a  vector of real-valued functions,
$a=\diag(a_1,\dots,a_n)$ and $b=\{b_{jk}\}_{j,k=1}^n$ are matrices of real-valued functions, and
$0\le m\le n$ are fixed integers.

Set
$$
\Pi = \{(x,t)\,:\,0<x<1, -\infty<t<\infty\}.
$$
Assume that there exists
$\eps_0 > 0$ such that for all $\eps\le\eps_0$ and all $(x,t)\in\overline\Pi$ the
following conditions are fulfilled:
\begin{equation}\label{eq:dif}
a_j, b_{jk}, p_{jk}, q_{jk} \mbox{  are  continuously differentiable in } x, t, \eps \mbox{ for all }
j,k\le n,
\end{equation}
\begin{equation}\label{eq:L1}
a_j>0 \mbox{ for all } j\le m\quad \mbox{ and }\quad a_j<0 \mbox{ for all } j>m,
\end{equation}
\begin{equation}\label{eq:L2}
\inf_{x,t}|a_j|>0
\quad \mbox { for all } j\le n,
\end{equation}
\begin{equation}\label{eq:L3}
\sup_{x,t}\left\{|a_j|,|\d_xa_j|,|\d_ta_j|,|\d_\eps a_j|\right\}<\infty \quad\mbox { for all } j\le n,
\end{equation}
\begin{equation}\label{eq:L5}
\sup_{x,t}\left\{|p_{jk}|,|q_{jk}|,|b_{jk}|,|\d_\eps b_{jk}|,|\d_tb_{jk}|
\right\}<\infty \quad\mbox { for all } j,k\le n
\end{equation}
\beq
\label{cass}
\begin{array}{l}
\mbox{for all } 1 \le j \not= k \le n \mbox{ there exist }  
\be_{jk},\ga_{jk}
\in C^1\left([0,1]\times \R\times[0,\eps_0)\right)
\\\mbox{such that }
b_{jk}(x,t,0)=\be_{jk}(x,t,\eps)\left(a_k(x,t,\eps)-a_j(x,t,0)\right)\\\mbox{and }
b_{jk}(x,t,\eps)=\ga_{jk}(x,t,\eps)\left(a_k(x,t,\eps)-a_j(x,t,\eps)\right),
\end{array}
\ee
and
\begin{equation}\label{eq:L4}
\begin{array}{l}
\displaystyle
\sup_{x,t}\left\{|\d_x\be_{jk}|,|\d_t\be_{jk}|,|\d_x\ga_{jk}|,|\d_t\ga_{jk}|
\right\}<\infty \quad\mbox { for all }  j\ne k.
\end{array}
\end{equation}

Given $s\in\R$, set
$$
\Pi_{s} = \{(x,t)\,:\,0<x<1, s<t<\infty\}.
$$
We subject the system \reff{eq:1}--\reff{eq:3} by the initial conditions at time $t=s$:
\begin{eqnarray}
u(x,s) = \varphi(x), \quad x\in[0,1],
\label{eq:2}
\end{eqnarray}
and consider the initial-boundary value problem (\ref{eq:1}), (\ref{eq:3}), (\ref{eq:2}) in $\Pi_{s}$
for arbitrarily fixed $s\in\R$.
Now we intend to switch to a weak formulation of the latter using integration along characteristic curves:
 For given $j\le n$, $x \in [0,1]$, $t \in \R$, and $\eps\in[0,\eps_0]$ the $j$-th characteristic of \reff{eq:1}
passing through the point $(x,t)\in\overline\Pi_s$ is defined
as the solution $\xi\in [0,1] \mapsto \om_j(\xi;x,t,\eps)\in \R$ of the initial value problem
\beq\label{char}
\partial_\xi\om_j(\xi;x,t,\eps)=\frac{1}{a_j(\xi,\om_j(\xi;x,t,\eps),\eps)},\;\;
\om_j(x;x,t,\eps)=t.
\ee
Write
\begin{eqnarray*}
c_j(\xi,x,t,\eps)=\exp \int_x^\xi
\left(\frac{b_{jj}}{a_{j}}\right)(\eta,\om_j(\eta;x,t,\eps),\eps)\,d\eta,\quad
d_j(\xi,x,t,\eps)=\frac{c_j(\xi,x,t,\eps)}{a_j(\xi,\om_j(\xi;x,t,\eps),\eps)}.
\end{eqnarray*}
Due to \reff{eq:L2}, the characteristic curve $\tau=\om_j(\xi;x,t,\eps)$ reaches the
boundary of $\Pi_s$ in two points with distinct ordinates. Let $x_j(x,t,\eps)$
denote the abscissa of that point whose ordinate is smaller.
Let us introduce linear bounded operators  $R: C\left(\overline\Pi_s\right)^n\mapsto C\left([s,\infty)\right)^n$ and
$B^\eps: C\left(\overline\Pi_s\right)^n\mapsto C\left(\overline\Pi_s\right)^n$ and an affine bounded operator
$S: C\left(\overline\Pi_s\right)^n\mapsto C\left(\overline\Pi_s\right)^n$ by
\begin{eqnarray}
\begin{array}{ll}
\displaystyle (Ru)_j(t)=\sum\limits_{k=m+1}^np_{jk}(t)u_k(0,t)+\sum\limits_{k=1}^mp_{jk}(t)u_k(1,t),\quad 1\le j\le m,
\\
\displaystyle (Ru)_j(t)= \sum\limits_{k=1}^mq_{jk}(t)u_k(0,t)+ \sum\limits_{k=m+1}^nq_{jk}(t)u_k(1,t), \quad m<j\le n,
\end{array}\label{eq:R}
\end{eqnarray}
\begin{eqnarray}
\label{B}
(B^\eps u)_j(x,t)=c_j(x_j(x,t,\eps),x,t,\eps)
u_j\left(x_j(x,t,\eps),\om_j(x_j(x,t,\eps);x,t,\eps)\right),
\end{eqnarray}
and
\begin{eqnarray}
\label{S}
(Su)_j(x,t)=
\left\{\begin{array}{lcl}
(Ru)_j(t) & \mbox{if}&  t>s, \\
\vphi_j(x)      & \mbox{if}& t=s.
\end{array}\right.
\end{eqnarray}
By abuse of notation, we did not indicate the dependence of the above operators on $s$;
in fact, in the consideration below the value of $s\in\R$ will be arbitrarily fixed.

Straightforward calculations show that a $C^1$-map
$u: [0,1]\times[0,\infty) \to \R^n$ is a solution to
(\ref{eq:1}), (\ref{eq:3}), (\ref{eq:2})  if and only if
it satisfies the following system of integral equations
\begin{eqnarray}
\label{rep}
\lefteqn{
u_j(x,t)=(B^\eps Su)_j(x,t)}\nonumber\\
&&-\int_{x_j(x,t,\eps)}^x d_j(\xi,x,t,\eps)
\sum_{k=1\atop k\not=j}^n b_{jk}(\xi,\om_j(\xi;x,t,\eps),\eps)
u_k(\xi,\om_j(\xi;x,t,\eps))d\xi,\quad  j\le n.
\end{eqnarray}
Now, the notion of weak (continuous) solution in $\Pi_s$ can  be naturally defined as follows.
\begin{defn}\label{defn:cont}\rm
 A continuous  function $u$ is called a continuous solution to  (\ref{eq:1}), (\ref{eq:3}), (\ref{eq:2})
in $\overline\Pi_s$ if it satisfies \reff{rep}.
\end{defn}

For given $\eps>0$, denote by $U^\varepsilon(t,s):  C([0,1])^n\mapsto  C([0,1])^n$ the evolution operator of the system
 (\ref{eq:1})--(\ref{eq:2}) whose existence is given by Theorem \ref{evol_oper}, i.e,
a bounded operator mapping the values of solutions at time $s$ into their values at time $t$
and satisfying the properties $U^\varepsilon (s,s) =I$ and $U^\varepsilon (t,s)U^\varepsilon (s,\tau) =  U^\varepsilon (t,\tau)$
for all $t \ge s \ge\tau$.

We examine robustness of
exponential dichotomies for a range of boundary operators  ensuring
that smoothness of   solutions increases in  finite time. With this aim we will assume that
the system  (\ref{eq:1})--(\ref{eq:3})   has a smoothing property
studied in \cite{kmit,Km}.

\begin{defn}\label{defn:smoothing}\rm
Let $\eps>0$.
The evolution operator $U^\eps(t,s)$
 to the problem (\ref{eq:1}), (\ref{eq:3})
is called {\it smoothing}
if, for every $s\in\R$, there exists $t>s$ such that
$U^\eps(t,s)\vphi\in C^1\left([0,1]\right)^n$
for every $\vphi\in C\left([0,1]\right)^n$.
\end{defn}

In the following definition  the range of an operator $P$ will be denoted by  $\Im P$.

\begin{defn}\label{defn:dich}\rm
Let $\eps>0$. We say that the system  (\ref{eq:1})--(\ref{eq:3}) has an exponential dichotomy on $\R$
with exponent $\beta > 0$ and bound $M$ if there exist projections $P^\varepsilon(t), t\in\R,$
such that

(i) $U^\eps(t,s)P^\varepsilon(s) = P^\varepsilon(t)U^\varepsilon(t,s), \ t \ge s$;

(ii) $U^\varepsilon(t,s)|_{\Im(P^\varepsilon(s))}$ for  $t \ge s$ is an isomorphism
on $\Im(P^\varepsilon(s))$, then
$U^\varepsilon(s,t)$ is defined as an inverse map from $\Im(P^\varepsilon(t))$
to $\Im(P^\varepsilon(s))$;

(iii) $\|U^\varepsilon(t,s)(1 - P^\varepsilon(s))\| \le M e^{-\beta(t-s)}, \ t \ge s$;

(iv) $\|U^\varepsilon(t,s)P^\varepsilon(s)\| \le M e^{\beta(t-s)}, \ t \le s$.
\end{defn}
Here and below by $\|\cdot\|$ we denote the operator norm in $\LL\left(C([0,1])^n \right)$.
\bigskip

Now we formulate our main result.

\begin{thm}\label{robust}
Suppose that the system (\ref{eq:1})--(\ref{eq:3}) with $\eps=0$ has an exponential dichotomy
and the corresponding evolution operator $U^0(t,s)$ is bounded:
\beq\label{evol_bound}
\sup\limits_{0\le t-s\le 1}\|U^0(t,s)\|<\infty.
\ee
Moreover, assume that there is $\eps_0>0$ such that the following conditions are fulfilled:
 \reff{eq:dif}--\reff{eq:L4} and
\beq\label{U-smooth}
\begin{array}{ll}
\mbox{ There exists } k\in\N \mbox{ such that } \left(B^\eps R\right)^k= 0
\mbox{ for all } \eps\le\eps_0.
\end{array}
\ee
Then there exists $\eps^\prime \le \eps_0$ such that
 for all  $\eps\le \eps^\prime$
the system (\ref{eq:1}), (\ref{eq:3}) has an exponential dichotomy.
\end{thm}

\begin{rem}\rm
 Note  that the boundary conditions \reff{eq:3}
together with the property \reff{U-smooth} generalize boundary conditions
appearing in  models of chemical kinetics \cite{AkrBelZel,Zel1}.
\end{rem}


\section{Basic facts}

The first fact follows from the results obtained in \cite{ijdsde,KmHo} and entails, in particular,
the existence of an evolution operator.

\begin{thm}\label{evol_oper}
Under the conditions  \reff{eq:dif}--\reff{eq:L4}, for
given $\eps>0$, $s\in\R$, $T>0$, and $\vphi\in C\left([0,1]\right)^n$ fulfilling the zero-order
compatibility conditions
\begin{eqnarray}
\begin{array}{ll}
\vphi_j(0) = (R\vphi)_j(s),\quad 1\le j\le m
\\
\vphi_j(1) = (R\vphi)_j(s), \quad m<j\le n,
\end{array}\label{comp}
\end{eqnarray}
the initial-boundary value problem (\ref{eq:1}), (\ref{eq:3}), (\ref{eq:2})  has a unique continuous
solution in $\Pi_s$ and this solution satisfies the apriori estimate
\beq\label{apriori}
\|u\|_{C\left(\overline\Pi_s\setminus\Pi_{s+T}\right)^n}\le C(T)\|\vphi\|_{C\left([0,1]\right)^n}
\ee
with a constant $C(T)>0$ depending on $T$, but  not on $s$, $\vphi$, and $\eps\le\eps_0$.
\end{thm}

The second fact can be readily obtained by \cite[Theorem 2.7]{Km} and the argument used in its proof.
It states the smoothing property of the evolution operator as well as the fact that
the time at which the continuous solution to (\ref{eq:1}), (\ref{eq:3}), (\ref{eq:2})
reaches the $C^1$-regularity does not exceed a fix number $d$, whatsoever initial time $s\in\R$.

\begin{lemma}\label{lem:d}
Under the conditions  \reff{eq:dif}--\reff{eq:L4} and \reff{U-smooth} the evolution operator is smoothing
and satisfies the following property:
\beq\label{d}
\begin{array}{ll}
\mbox{ There exists } d>0  \mbox{ such that for any } s\in\R  \mbox{ and } t
\mbox{ as in Definition \ref{defn:smoothing}}, \\\mbox{ the inequality }
|t-s|\le d
 \mbox{ is true for all }
 \eps\le\eps_0.
\end{array}
\ee
\end{lemma}

The third fact is a variant of  \cite[Theorem 7.6.10]{H}.

\begin{thm}\label{robust_gen}
Assume that the evolution operator $U^0(t,s)$ has an exponential dichotomy on
 $\R$ and satisfies \reff{evol_bound}.
Then there exists $\eta > 0$ such that for all $\eps>0$ with
$$
\|U^0(t,s) - U^\varepsilon(t,s)\| < \eta, \ \ {\rm whenever} \ \ t-s = 2d
$$
the evolution operator $U^\eps(t,s)$ has an exponential dichotomy on
 $\R$ also.
\end{thm}

\begin{proof}
Given $s\in\R$ and $\eps>0$, set
 $$
t_n = s + 2dn,\quad T_n^\eps = U^\eps(t_0 + 2d(n+1), t_0 + 2dn)\qquad \mbox{ for } n \in \Z.
$$
If the evolution operator $U^0(t,s)$ has an exponential dichotomy,
then  the sequence $\left\{T_n^0\right\}$ has a discrete dichotomy in the sense of \cite[Definition 7.6.4]{H}.

By \cite[Theorem 7.6.7]{H}, there exists $\eta> 0$ such that for all $\eps>0$
with
$$
\sup_n\|T_n^0 - T_n^\eps\|  \le \eta
$$
$\left\{T_n^\eps\right\}$ has a discrete dichotomy.

Now we are in the conditions of \cite[Excersise 10, p. 229--230]{H} (see also a more general statement \cite[Theorem 4.1]{ChL1}),
what finishes the proof.
\end{proof}


\section{Proof of Theorem~\ref{robust}}\label{sec:main_thm}

Given $s\in\R$ and $\eps>0$, let us introduce
 linear bounded operators $D^\eps, F^\eps: C\left(\overline\Pi_s\right)^n \to
C\left(\overline\Pi_s\right)^n$
 by
\begin{eqnarray*}
\left(D^\eps w\right)_j(x,t) & = &
-\int_{x_j(x,t,\eps)}^x d_j(\xi,x,t,\eps)\sum_{k=1\atop k\not=j}^nb_{jk}(\xi,\om_j(\xi;x,t,\eps),\eps)
w_k(\xi,\om_j(\xi;x,t,\eps))\,d\xi,\\
\left(F^\eps f\right)_j(x,t)&=&\int_{x_j(x,t,\eps)}^x d_j(\xi,x,t,\eps)f_j(\xi,\om_j(\xi;x,t,\eps))d\xi.
\end{eqnarray*}
Here again we dropped the dependence of $D^\eps$ and $F^\eps$ on $s$, as throughout the proof
$s\in\R$ is arbitrarily fixed.
To simplify further notation, set
\beq
\begin{array}{cc}
a(x,t)=a(x,t,0),\quad b(x,t)=b(x,t,0),  c_j(x,t)=c_j(x,t,0),\quad d_j(x,t)=d_j(x,t,0),\quad \\
a^\eps(x,t)=a(x,t,\eps),\quad b^\eps(x,t)=b(x,t,\eps),\quad \be_{jk}^\eps(x,t)=\be_{jk}(x,t,\eps),\\
\om_j(\xi;x,t)=\om_j(\xi;x,t,0),\quad x_j(x,t)= x_j(x,t,0),\quad D=D^0,\quad F=F^0.
\end{array}
\ee

Fix arbitrary  values $s\in\R$ and $\eps\le\eps_0$ and an arbitrary initial function  
$\vphi\in C\left([0,1]\right)^n$
in (\ref{eq:2}). Let $u$ and $v$ be the continuous solutions to the problem (\ref{eq:1}), (\ref{eq:3}), (\ref{eq:2})
with $\eps=0$ and $\eps$, respectively.
By Lemma~\ref{lem:d}, the evolution operator $U^\eps(t,s)$ is smoothing with the time of smoothing
not exceeding $d$. This means that starting at $t=s+d$ the
solutions $u$ and $v$ are continuously differentiable and, therefore, satisfy the system
(\ref{eq:1})  pointwise. Hence, the difference $u-v$
fulfills the equation
\begin{equation}\label{eq:1u}
(\partial_t  + a(x,t)\partial_x + b(x,t))
 (u-v) = \left(a^\eps(x,t)-a(x,t)\right)\partial_xv+ \left(b^\eps(x,t)-b(x,t)\right)v,\quad (x,t)\in\Pi_{s+d}
\end{equation}
and the boundary conditions
\begin{eqnarray}
\begin{array}{ll}
(u_j-v_j)(0,t) = \left(R(u-v)\right)_j(t),\quad 1\le j\le m,\quad t\ge s
\\
(u_j-v_j)(1,t) =  \left(R(u-v)\right)_j(t), \quad m<j\le n,\quad t\ge s,
\end{array}\label{eq:3u}
\end{eqnarray}
or, the same, the operator equation
\begin{eqnarray}
\label{rep_(u-v)}
u-v\big|_{\overline\Pi_{s+d}}=BR(u-v)+D(u-v)+F\left(\left(a^\eps-a\right)\partial_xv\right)+F\left(\left(b^\eps-b\right)v\right).
\end{eqnarray}
A similar equation is true for $u-v$ under the operator $BR$, what entails
\begin{eqnarray}
u-v\big|_{\overline\Pi_{s+d}}&=&(BR)^2(u-v)+(I+BR)D(u-v)\nonumber\\
&+&(I+BR)F\left(\left(a^\eps-a\right)\partial_xv\right)+(I+BR)F\left(\left(b^\eps-b\right)v\right).\nonumber
\end{eqnarray}
Doing this iteration, on the $k$-th step we meet the property (see \reff{U-smooth})
\beq\label{U-smooth1}
\left(BR\right)^{k}(u-v)\equiv 0
\ee
and, hence, get the formula
\begin{eqnarray*}
u-v\big|_{\overline\Pi_{s+d}}=\sum_{i=0}^{k-1}(BR)^{i}D(u-v)+
\sum_{i=0}^{k-1}(BR)^{i}F\left(\left(a^\eps-a\right)\partial_xv\right)+
\sum_{i=0}^{k-1}(BR)^{i}F\left(\left(b^\eps-b\right)v\right).\nonumber
\end{eqnarray*}
In particular,
\begin{eqnarray}
\label{rep_(u-v)_smooth}
\lefteqn{
(u-v)(x,s+2d)=\sum_{i=0}^{k-1}\left[(BR)^{i}D(u-v)\right](x,s+2d)}\\
&&+
\sum_{i=0}^{k-1}\left[(BR)^{i}F\left(\left(a^\eps-a\right)\partial_xv\right)\right](x,s+2d)
+
\sum_{i=0}^{k-1}\left[(BR)^{i}F\left(\left(b^\eps-b\right)v\right)\right](x,s+2d).\nonumber
\end{eqnarray}
Therefore, on the account of  Theorem \ref{robust_gen}, we are done if we show that, given  $\eta>0$,
there is $\eps^\prime\le\eps_0$ such that
\beq\label{main_estim}
\|(u-v)(\cdot,s+2d)\|_{C\left([0,1]\right)^n}\le \eta\|\vphi\|_{C\left([0,1]\right)^n},
\ee
the bound being uniform in $s\in\R$, $\eps\le\eps^\prime$, and $\vphi\in C\left([0,1]\right)^n$.
To derive \reff{main_estim}, we estimate
each of the three sums in  \reff{rep_(u-v)_smooth} separately.

To obtain the desired estimate for the first sum in \reff{rep_(u-v)_smooth},
 we first derive the formula for $D(u-v)$ contributing into this summand.
To this end, use the operator representation  for $u$ and $v$, namely,
$$
u=BSu+Du,\quad v=B^\eps Sv+D^\eps v,
$$
where the functions $u$ and $v$ are restricted to $\overline\Pi_s\setminus\Pi_{s+2d}$
and the operators $B^\eps, S$, and $D^\eps $ are restricted to $C\left(\overline\Pi_s\setminus\Pi_{s+2d}\right)^n$.
Note that, as it follows from the definition,  $B^\eps, S$, and $D^\eps $ map  
$C\left(\overline\Pi_s\setminus\Pi_{s+2d}\right)^n$ into $C\left(\overline\Pi_s\setminus\Pi_{s+2d}\right)^n$.
Thus, for the difference we have
\begin{eqnarray}
\label{rep_(u-v)0}
u-v=BS(u-v)+\left(B-B^\eps\right)Sv+D(u-v)+(D-D^\eps)v,
\end{eqnarray}
hence
\begin{eqnarray}
\label{rep_(u-v)1}
D(u-v)=DBS(u-v)+D\left(B-B^\eps\right)Sv+D^2(u-v)+D(D-D^\eps)v.
\end{eqnarray}
Our next objective is to rewrite the last equation with respect to the new variable
\beq\label{w}
w=D(u-v).
\ee
With this aim we substitute \reff{rep_(u-v)0} into the first summand in the right-hand side of \reff{rep_(u-v)1} and get
\begin{eqnarray}
\label{rep_(u-v)2}
w&=&D(BS)^2(u-v)+D(I+BS)\left(B-B^\eps\right)Sv\nonumber\\
&+&D(I+BS)w+D(I+BS)(D-D^\eps)v.
\end{eqnarray}
Continuing in this fashion (again substituting \reff{rep_(u-v)0} into the first summand
in the right-hand side of \reff{rep_(u-v)2}), on the $r$-th step  we arrive at the formula
\begin{eqnarray}
\label{rep_(u-v)3}
w&=&D(BS)^r(u-v)+D\sum_{i=0}^{r-1}(BS)^i\left(B-B^\eps\right)Sv\nonumber\\
&+&D\sum_{i=0}^{r-1}(BS)^iw+D\sum_{i=0}^{r-1}(BS)^i(D-D^\eps)v.
\end{eqnarray}
Since $(u-v)(\cdot,s)\equiv 0$ on $[0,1]$, there exists $r_0\in\N$
 such that
 $(BS)^{r_0}(u-v)=0$. 
Therefore, the resulting equation for $w$ restricted to $\overline\Pi_s\setminus\Pi_{s+2d}$
can be written as
\begin{eqnarray}
\label{rep_(u-v)4}
w=D\sum_{i=0}^{r_0-1}(BS)^i\left(B-B^\eps\right)Sv+D\sum_{i=0}^{r_0-1}(BS)^i(D-D^\eps)v
+D\sum_{i=0}^{r_0-1}(BS)^iw.
\end{eqnarray}
Our goal now is to show the existence of a function
$\al:[0,1]\to\R$ with $\al(\eps)\to 0$ as $\eps\to 0$ for which we have
\begin{eqnarray}
\label{apr_w}
\|w\|_{C\left(\overline\Pi_s\setminus\Pi_{s+2d}\right)^n}\le\al(\eps)\|\vphi\|_{C([0,1])^n},
\end{eqnarray}
the estimate being uniform in   $s\in\R$ and $\vphi\in C([0,1])^n$ satisfying the zero-order compatibility conditions
\reff{comp}.
With this aim we first
 show that there is  a function
$\tilde\al(\eps)$  meeting the same properties as $\al(\eps)$ such that the
$C\left(\overline\Pi_s\setminus\Pi_{s+2d}\right)^n$-norm
of the first two summands in the right-hand side of \reff{rep_(u-v)4} is bounded from above by
$\tilde\al(\eps)\|\vphi\|_{C([0,1])^n}$.
Afterwords, we use the boundedness of the operators
$B,S,D$ restricted to $C\left(\overline\Pi_s\setminus\Pi_{s+2d}\right)^n$,
then apply Gronwall's inequality to \reff{rep_(u-v)4}, and this way
derive \reff{apr_w}. To this end, observe that  the integral operator $D$ can be considered as
Volterra  operator of the second kind. This follows from the fact that $D$ can be equivalently defined by the formula
\begin{eqnarray*}
\left(D w\right)_j(x,t) & = &
-\int_{t_j(x,t)}^t \tilde d_j(\tau,x,t)\sum_{k=1\atop k\not=j}^nb_{jk}(\tilde\om_j(\tau;x,t),\tau)
w_k(\tilde\om_j(\tau;x,t),\tau)\,d\tau,
\end{eqnarray*}
where $\tau\in\R \mapsto \tilde\om_j(\tau;x,t)\in[0,1]$ is the inverse form of
the $j$-th characteristic of \reff{eq:1}
passing through the point $(x,t)\in\overline\Pi$, $t_j(x,t)$ is a minimum value of $\tau$
at which the characteristic $\tau=\tilde\om_j(\tau;x,t)$ reaches $\d\Pi_s$, and
\begin{eqnarray*}
\tilde d_j(\tau,x,t)=\exp \int_t^\tau
b_{jj}(\tilde\om_j(\eta;x,t),\eta)\,d\eta.
\end{eqnarray*}

Thus, the estimate \reff{apr_w} will be proved as soon as we derive the upper bound
$\tilde\al(\eps)\|\vphi\|_{C([0,1])^n}$ for the absolute value of the first two summands in \reff{rep_(u-v)4}.
The idea behind the proof is a smoothing property
of the operators representing those summands. We prove this only for one summand in each sum (when $i=0$).
For all other summands we apply similar argument.

Thus, to get the desired estimate for the summand $D(D-D^\eps)v$, it suffices to show that, given $j\le n$,
the function $\left(DD^\eps v\right)_j(x,t)$ is continuously differentiable in $\eps$ and that the derivative
is bounded on $\overline\Pi_s\setminus\Pi_{s+2d}$ uniformly in $s\in\R$ and $\eps\le\eps_0$.
Indeed, following the techniques from  \cite{kmit}, fix a sequence  $v^l\in C^1\left(\overline\Pi\right)^n$ such that
\beq\label{eq:lim_0}
 v^l\to v \mbox{ in } C\left(\overline\Pi_s\setminus \Pi_{s+2d}\right)^n \mbox{ as } l\to\infty.
\ee
We are done if we prove
that $\d_\eps\left[\left(DD^\eps v^l\right)_j(x,t)\right]$ converges in  $C\left(\overline\Pi_s\setminus \Pi_{s+2d}\right)$ as
$l\to\infty$ and that the limit function is bounded on $\overline\Pi_s\setminus\Pi_{s+2d}$ uniformly in $s\in\R$
and $\eps\le\eps_0$.
Consider the following expression for $\left(DD^\eps v^l\right)_j(x,t)$:
\begin{eqnarray}
\left(DD^\eps v^l\right)_j(x,t)
&=&\sum_{k=1\atop k\not=j}^n\sum_{i=1\atop i\not=k}^n
\int_{x_j(x,t)}^x \int_{x_k(\xi,\om_j(\xi;x,t),\eps)}^\xi d_{jki}(\xi,\eta,x,t,\eps)b_{jk}(\xi,\om_j(\xi;x,t))\nonumber\\
&\times&
v_i^l(\eta,\om_k(\eta;\xi,\om_j(\xi;x,t),\eps))\, d \eta d \xi\label{D11}
\end{eqnarray}
with
\begin{eqnarray*}
d_{jki}(\xi,\eta,x,t,\eps)
=d_j(\xi,x,t)d_k(\eta,\xi,\om_j(\xi;x,t),\eps)b_{ki}(\eta,\om_k(\eta;\xi,\om_j(\xi;x,t),\eps),\eps).
\label{djkl}
\end{eqnarray*}
 Let $x_{jk}(x,t,\eps)$ denote the $x$-coordinate
of the point (if any) where the characteristics $\om_j(\xi;x,t)$ and $\om_k(\xi;0,s,\eps)$ if $k\le m$
and the characteristics $\om_j(\xi;x,t,\eps)$ and $\om_k(\xi;1,s,\eps)$ if $k>m$ intersect. Hence, $x_{jk}(x,t,\eps)$
satisfies the  equation
\begin{eqnarray}
\label{x_jk1}
\om_j(x_{jk}(x,t,\eps);x,t)=\om_k(x_{jk}(x,t,\eps);0,s,\eps)
\end{eqnarray}
if $k\le m$ and the equation
\begin{eqnarray}
\label{x_jk}
\om_j(x_{jk}(x,t,\eps);x,t)=\om_k(x_{jk}(x,t,\eps);1,s,\eps)
\end{eqnarray}
if $k>m$. Suppose for definiteness that $j\le m$ and $k>m$ (similar argument works for all other $j\ne k$). Thus, if $x_{jk}(x,t,\eps)$ exists for some $(x,t,\eps)$,
then the integrals in  \reff{D11} admit the decomposition
\begin{eqnarray}
\label{decom}
\int_{x_j(x,t)}^x \int_{x_k(\xi,\om_j(\xi;x,t),\eps)}^\xi d \eta d \xi=
\int_{x_j(x,t)}^{x_{jk}(x,t,\eps)} \int_{x_k(\xi,\om_j(\xi;x,t),\eps)}^\xi d \eta d \xi+
\int_{x_{jk}(x,t,\eps)}^x \int_1^\xi d \eta d \xi,
\end{eqnarray}
where the function $x_k(\xi,\om_j(\xi;x,t),\eps)$ in the right-hand side  satisfies the equality
\begin{eqnarray}
\label{s}
\om_k(x_k(\xi,\om_j(\xi;x,t),\eps);\xi,\om_j(\xi;x,t),\eps)=s.
\end{eqnarray}
Now we intend to show that the derivatives  $\d_\eps x_k(\xi,\om_j(\xi;x,t),\eps)$ and $\d_\eps x_{jk}(x,t,\eps)$ exist.
With this aim we introduce a couple of useful formulas:
\begin{eqnarray}
\label{dx}
\d_x\om_j(\xi;x,t,\eps) & = & -\frac{1}{a_j(x,t,\eps)} \exp \int_\xi^x
\left(\frac{\d_ta_j}{a_j^2}\right)(\eta,\om_j(\eta;x,t,\eps),\eps) d \eta,
\\
\label{dt}
\d_t\om_j(\xi;x,t,\eps) & = & \exp \int_\xi^x \left(\frac{\d_ta_j}{a_j^2}\right)(\eta,\om_j(\eta;x,t,\eps),\eps) d \eta,\\
\label{d_eps}
\d_\eps\om_j(\xi;x,t,\eps) & = & \exp \int_\xi^x \left(\frac{\d_ta_j}{a_j^2}\right)(\eta,\om_j(\eta;x,t,\eps),\eps) d \eta
\nonumber\\\nonumber
&\times&
\int_\xi^x \left(\frac{\d_\eps a_j}{a_j^2}\right)(\eta,\om_j(\eta;x,t,\eps),\eps)\\
&\times&
 \exp \int^{\eta}_x \left(\frac{\d_ta_j}{a_j^2}\right)(\eta_1,\om_j(\eta_1;x,t,\eps),\eps)\, d \eta_1\,d\eta.
\end{eqnarray}
Then the existence of the derivatives  $\d_\eps x_k(\xi,\om_j(\xi;x,t),\eps)$ and $\d_\eps x_{jk}(x,t,\eps)$
follow from the equalities \reff{s} and \reff{x_jk}, respectively. Furthermore, we derive the formulas
\beq
\label{d_x_k}
\d_\eps x_k(\xi,\om_j(\xi;x,t),\eps)= -a_k(x_k(\xi,\om_j(\xi;x,t),\eps),s)
\d_4\om_k(x_k(\xi,\om_j(\xi;x,t),\eps);\xi,\om_j(\xi;x,t),\eps)
\ee
and
\beq
\label{d_x_jk}
\d_\eps x_{jk}(x,t,\eps)\left(\frac{a_k^\eps-a_j}{a_ja_k^\eps}\right)(x_{jk}(x,t,\eps),\om_j(x_{jk}(x,t,\eps);x,t))=
\d_4\om_k(x_{jk}(x,t,\eps);1,s,\eps).
\ee
Hence, on the account of the assumption \reff{cass}, from the last equality we get
\beq
\label{bd_x_jk}
\begin{array}{cc}
\d_\eps x_{jk}(x,t,\eps)b_{jk}(x_{jk}(x,t,\eps),\om_j(x_{jk}(x,t,\eps);x,t))\\
=
\left(\be_{jk}^\eps a_ja_k^\eps\right)(x_{jk}(x,t,\eps),\om_j(x_{jk}(x,t,\eps);x,t))\d_4\om_k(x_{jk}(x,t,\eps);1,s,\eps).
\end{array}
\ee

Now, using the regularity assumption \reff{eq:dif}, we are able to compute the derivative 
\begin{eqnarray}
\lefteqn{
\d_\eps\left[\left(DD^\eps v^l\right)_j(x,t)\right]
}\nonumber\\
&=&\sum_{k=1\atop k\not=j}^n\sum_{i=1\atop i\not=k}^n
\int_{x_j(x,t)}^x \int_{x_k(\xi,\om_j(\xi;x,t),\eps)}^x \d_\eps\Bigl[d_{jki}(\xi,\eta,x,t,\eps)
b_{jk}(\xi,\om_j(\xi;x,t))\Bigr]\nonumber\\
&\times&
v_i^l(\eta,\om_k(\eta;\xi,\om_j(\xi;x,t),\eps)) \,d \eta d \xi \nonumber\\
&+&\sum_{k=1\atop k\not=j}^n\sum_{i=1\atop i\not=k}^n
\int_{x_j(x,t)}^x \int_{x_k(\xi,\om_j(\xi;x,t),\eps)}^x {d}_{jki}(\xi,\eta,x,t,\eps)b_{jk}(\xi,\om_j(\xi;x,t))\nonumber\\
&\times&\d_\eps\om_k(\eta;\xi,\om_j(\xi;x,t),\eps)\d_2
v_i^l(\eta,\om_k(\eta;\xi,\om_j(\xi;x,t),\eps))\, d \eta d \xi\nonumber\\
&+&\sum_{k=1\atop k\not=j}^n\sum_{i=1\atop i\not=k}^n
\left(\be_{jk}^\eps a_ja_k^\eps\right)(x_{jk}(x,t,\eps),\om_j(x_{jk}(x,t,\eps);x,t))\d_4\om_j(x_{jk}(x,t,\eps);1,s,\eps)
\nonumber\\
&\times&
\int_{x_k(x_{jk}(x,t,\eps),\om_j(x_{jk}(x,t,\eps);x,t),\eps)}^\xi
\Bigl[d_{jki}(\xi,\eta,x,t,\eps)v_i^l(\eta,\om_k(\eta;\xi,\om_j(\xi;x,t),\eps))\Bigr]_{\xi=x_{jk}(x,t,\eps)}\, d \eta
\nonumber\\
&-&\sum_{k=1\atop k\not=j}^n\sum_{i=1\atop i\not=k}^n
\int_{x_j(x,t)}^{x_{jk}(x,t,\eps)}
\d_\eps x_k(\xi,\om_j(\xi;x,t),\eps)
b_{jk}(\xi,\om_j(\xi;x,t))\nonumber\\
&\times&
\Bigl[d_{jki}(\xi,\eta,x,t,\eps)
v_i^l(\eta,\om_k(\eta;\xi,\om_j(\xi;x,t),\eps))
\Bigr]_{\eta=x_k(\xi,\om_j(\xi;x,t),\eps)} \,d \xi\nonumber\\
&-&\sum_{k=1\atop k\not=j}^n\sum_{i=1\atop i\not=k}^n
\left(\be_{jk}^\eps a_ja_k^\eps\right)(x_{jk}(x,t,\eps),\om_j(x_{jk}(x,t,\eps);x,t))\d_4\om_j(x_{jk}(x,t,\eps);1,s,\eps)
\nonumber\\
&\times&
\int_1^\xi\Bigl[d_{jki}(\xi,\eta,x,t,\eps)v_i^l(\eta,\om_k(\eta;\xi,\om_j(\xi;x,t),\eps))\Bigr]_{\xi=x_{jk}(x,t,\eps)}\, d \eta,
\label{dtD}
\end{eqnarray}
where $\d_rg$ here and below  denotes the derivative of $g$ with respect to the $r$-th argument.
Note that $x_k(x_{jk}(x,t,\eps),\om_j(x_{jk}(x,t,\eps);x,t),\eps)=1$, hence the third and the fifth summands
in the right-hand side cancel out.
The first and the fourth summands  converge in $C\left(\overline\Pi_s\setminus \Pi_{s+2d}\right)$ as $l\to\infty$.
Our task is therefore reduced to show the uniform convergence
of all integrals in the second summand. For this purpose we will transform the integrals as follows:
Changing  the order of integration and using  \reff{eq:dif} and \reff{cass},
we get (to simplify notation in the calculation below  we drop the dependence of $x_j$ on $x$ and $t$)
\begin{eqnarray}
\lefteqn{\int_{x_j}^x \int_\eta^x {d}_{jki}(\xi,\eta,x,t,\eps)b_{jk}(\xi,\om_j(\xi;x,t))\d_\eps\om_k(\eta;\xi,\om_j(\xi;x,t),\eps)}\nonumber\\
&\times&\d_2v_i^l(\eta,\om_k(\eta;\xi,\om_j(\xi;x,t),\eps))\,
d \xi d \eta\nonumber\\
&=&\int_{x_j}^x \int_\eta^x {d}_{jki}(\xi,\eta,x,t,\eps)
\d_\eps\om_k(\eta;\xi,\om_j(\xi;x,t),\eps)b_{jk}(\xi,\om_j(\xi;x,t))\nonumber\\
&\times&\Bigl[\bigl(\d_\xi\om_k\bigr)(\eta;\xi,\om_j(\xi;x,t),\eps)\Bigr]^{-1}
\bigl(\d_\xi v_i^l\bigr)(\eta,\om_k(\eta;\xi,\om_j(\xi;x,t),\eps))d \xi d \eta\nonumber\\
&=&\int_{x_j}^x \int_\eta^x {d}_{jki}(\xi,\eta,x,t,\eps)\d_\eps\om_k(\eta;\xi,\om_j(\xi;x,t),\eps))
\Bigl(\be_{jk}^\eps a_ja_k^\eps\Bigr)(\xi,\om_j(\xi;x,t))\nonumber\\
&\times&\d_3\om_k(\eta;\xi,\om_j(\xi;x,t),\eps)
\bigl(\d_\xi v_i^l\bigr)(\eta,\om_k(\eta;\xi,\om_j(\xi;x,t),\eps)) d \xi d \eta\nonumber\\
&=&\int_{x_j}^x \int_\eta^x \tilde{d}_{jki}(\xi,\eta,x,t,\eps)\bigl(\d_\xi v_i^l\bigr)(\eta,\om_k(\eta;\xi,\om_j(\xi;x,t),\eps) )
d \xi d \eta\nonumber\\
\nonumber&=&-\int_{x_j}^x\int_\eta^x \d_\xi \tilde{d}_{jki}(\xi,\eta,x,t,\eps)v_i^l
\left(\eta,\om_k(\eta;\xi,\om_j(\xi;x,t),\eps)\right) d\xi d\eta\nonumber\\
&+&\int_{x_j}^x\Big[\tilde{d}_{jki}(\xi,\eta,x,t,\eps)v_i^l\left(\eta,\om_k(\eta;\xi,\om_j(\xi;x,t),\eps)\right)
\Big]_{\xi=\eta}^{\xi=x}d\eta.\label{by_parts}
\end{eqnarray}
Here
\begin{eqnarray*}
\tilde{d}_{jki}(\xi,\eta,x,t,\eps) &=& {d}_{jki}(\xi,\eta,x,t,\eps)\d_\eps\om_k(\eta;\xi,\om_j(\xi;x,t),\eps)\\
&\times&
\d_3\om_k(\eta;\xi,\om_j(\xi;x,t),\eps)\left(\be_{jk}^\eps a_ja_k^\eps
\right)
(\xi,\om_j(\xi;x,t)).
\end{eqnarray*}
Now, the desired convergence follows from  \reff{eq:lim_0}  and \reff{dx}--\reff{d_eps}.
The desired boundedness of the limit function is a consequence of the assumptions  \reff{eq:L3},
\reff{eq:L5}, and \reff{eq:L4}.

Returning to the formula \reff{rep_(u-v)4}, similar argument works also for the operators contributing into the first sum:
Again, for $i=0$, on the account of the definition of the operators $D$ and $B^\eps$,  we have
to show that the $\eps$-derivative of
\begin{eqnarray}
\lefteqn{
\left(DB^\eps v^l\right)_j(x,t)}\nonumber\\
&=&\sum_{k=1\atop k\not=j}^n
\int^{x_j(x,t)}_x  d_j(\xi,x,t)b_{jk}(\xi,\om_j(\xi;x,t))
c_k(x_k(\xi,\om_j(\xi;x,t),\eps),\xi,\om_j(\xi;x,t),\eps)\nonumber\\
&\times& v_k^l(x_k(\xi,\om_j(\xi;x,t),\eps),\om_k(x_k(\xi,\om_j(\xi;x,t),\eps);\xi,\om_j(\xi;x,t),\eps) d \xi
\label{DB}
\end{eqnarray}
converges uniformly on $\overline\Pi_s\setminus \Pi_{s+2d}$ and that the limit function is bounded uniformly in $s\in\R$
and $\eps\le\eps_0$. To show this, we  differentiate \reff{DB} in $\eps$, use \reff{cass}, and integrate by parts.
To be more precise, fix arbitrary $j\le m$ and $k>m$ (similarly for all other $j\ne k$) and rewrite the $k$-th summand in the right-hand side of \reff{DB} as (up to the sign)
\begin{eqnarray}
\label{DB1}
\lefteqn{
\int_{x_j(x,t)}^{x_{jk}(x,t,\eps)}  d_j(\xi,x,t)b_{jk}(\xi,\om_j(\xi;x,t))
c_k(x_k(\xi,\om_j(\xi;x,t),\eps),\xi,\om_j(\xi;x,t),\eps)}\nonumber\\
&\times& v_k^l(x_k(\xi,\om_j(\xi;x,t),\eps),s) d \xi
\nonumber\\
&+&
\int_{x_{jk}(x,t,\eps)}^x  d_j(\xi,x,t)b_{jk}(\xi,\om_j(\xi;x,t))
c_k(1,\xi,\om_j(\xi;x,t),\eps)\nonumber\\
&\times& v_k^l(1,\om_k(1;\xi,\om_j(\xi;x,t),\eps)) d \xi
\end{eqnarray}
Then the $\eps$-derivative of this expression equals
\begin{eqnarray}
\label{DB2}
\lefteqn{
\int_{x_j(x,t)}^{x}  d_j(\xi,x,t)b_{jk}(\xi,\om_j(\xi;x,t))
\d_\eps c_k(x_k(\xi,\om_j(\xi;x,t),\eps),\xi,\om_j(\xi;x,t),\eps)}\nonumber\\
&\times& v_k^l(x_k(\xi,\om_j(\xi;x,t),\eps),s) d \xi
\nonumber\\
&+&
\int_{x_j(x,t)}^{x_{jk}(x,t,\eps)}  d_j(\xi,x,t)b_{jk}(\xi,\om_j(\xi;x,t))
c_k(x_k(\xi,\om_j(\xi;x,t),\eps),\xi,\om_j(\xi;x,t),\eps)\nonumber\\
&\times& \d_\eps x_k(\xi,\om_j(\xi;x,t),\eps)\d_1v_k^l(x_k(\xi,\om_j(\xi;x,t),\eps),s) d \xi
\nonumber\\
&+&
\int_{x_{jk}(x,t,\eps)}^x  d_j(\xi,x,t)b_{jk}(\xi,\om_j(\xi;x,t))
c_k(1,\xi,\om_j(\xi;x,t),\eps)\nonumber\\
&\times& \d_\eps\om_k(1;\xi,\om_j(\xi;x,t),\eps)\d_2v_k^l(1,\om_k(1;\xi,\om_j(\xi;x,t),\eps)) d \xi.
\end{eqnarray}
For the first summand the desired convergence and the uniform boundedness of the limit function is obvious.
The last two  summands are equal to
\begin{eqnarray}
\label{DB3}
\lefteqn{
\int_{x_j(x,t)}^{x_{jk}(x,t,\eps)}  d_j(\xi,x,t)b_{jk}(\xi,\om_j(\xi;x,t))
c_k(x_k(\xi,\om_j(\xi;x,t),\eps),\xi,\om_j(\xi;x,t),\eps)}\nonumber\\
&\times& \d_\eps x_k(\xi,\om_j(\xi;x,t),\eps)
\left[\d_\xi x_k(\xi,\om_j(\xi;x,t),\eps)\right]^{-1}
\d_\xi v_k^l(x_k(\xi,\om_j(\xi;x,t),\eps),s) d \xi
\nonumber\\
&+&
\int_{x_{jk}(x,t,\eps)}^x  d_j(\xi,x,t)b_{jk}(\xi,\om_j(\xi;x,t))
c_k(1,\xi,\om_j(\xi;x,t),\eps)\d_\eps\om_k(1;\xi,\om_j(\xi;x,t),\eps)\nonumber\\
&\times&
\left[\d_\xi\om_k(1;\xi,\om_j(\xi;x,t),\eps))\right]^{-1}
\d_\xi v_k^l(1,\om_k(1;\xi,\om_j(\xi;x,t),\eps)) d \xi.
\end{eqnarray}
Next we use the formulas \reff{s}, \reff{dx}, and  \reff{dt} and calculate
\begin{eqnarray*}
\label{d_xi1}
\displaystyle\d_\xi x_k(\xi,\om_j(\xi;x,t),\eps)&=&a_k(x_k(\xi,\om_j(\xi;x,t),\eps),s)\left(\frac{a_k^\eps-a_j}{a_ja_k^\eps}\right)(\xi,\om_j(\xi;x,t))\\
 &\times&\displaystyle
\d_3\om_k(x_k(\xi,\om_j(\xi;x,t),\eps);\xi,\om_j(\xi;x,t),\eps),
\nonumber\\[3mm]
\displaystyle\d_\xi\om_k(1;\xi,\om_j(\xi;x,t),\eps)&=&\left(\frac{a_k^\eps-a_j}{a_ja_k^\eps}\right)(\xi,\om_j(\xi;x,t))
\d_3\om_k(1;\xi,\om_j(\xi;x,t),\eps).
\end{eqnarray*}
Now, due to the assumptions \reff{eq:dif} and \reff{cass}, we are in a position to bring the expression \reff{DB3}
to a desirable form
\begin{eqnarray}
\label{DB4}
\lefteqn{
\int_{x_j(x,t)}^{x_{jk}(x,t,\eps)}  d_j(\xi,x,t)
c_k(x_k(\xi,\om_j(\xi;x,t),\eps),\xi,\om_j(\xi;x,t),\eps)\d_\eps x_k(\xi,\om_j(\xi;x,t),\eps)}\nonumber\\
&\times&
a_k(x_k(\xi,\om_j(\xi;x,t),\eps),s)\d_3\om_k(x_k(\xi,\om_j(\xi;x,t),\eps);\xi,\om_j(\xi;x,t),\eps)\nonumber\\
&\times&
\left(a_ja_k^\eps\be_{jk}^\eps\right)(\xi,\om_j(\xi;x,t))
\d_\xi v_k^l(x_k(\xi,\om_j(\xi;x,t),\eps),s) d \xi
\nonumber\\
&+&
\int_{x_{jk}(x,t,\eps)}^x  d_j(\xi,x,t)
c_k(1,\xi,\om_j(\xi;x,t),\eps)\d_\eps\om_k(1;\xi,\om_j(\xi;x,t),\eps)\nonumber\\
&\times& \d_3\om_k(1;\xi,\om_j(\xi;x,t),\eps)\left(a_ja_k\be_{jk}^\eps\right)(\xi,\om_j(\xi;x,t))
\d_\xi v_k^l(1,\om_k(1;\xi,\om_j(\xi;x,t),\eps)) d \xi\nonumber\\
&=&-\int_{x_j(x,t)}^{x_{jk}(x,t,\eps)}  \d_\xi e_{jk}(\xi,x,t,\eps)
v_k^l(x_k(\xi,\om_j(\xi;x,t),\eps),s) d \xi\nonumber\\
&+& e_{jk}(\xi,x,t,\eps)
v_k^l(x_k(\xi,\om_j(\xi;x,t),\eps),s) \Big|^{x_{jk}(x,t,\eps)}_{\xi=x_j(x,t)}\nonumber\\
&-&
\int_{x_{jk}(x,t,\eps)}^x\d_\xi \tilde e_{jk}(\xi,x,t,\eps)
 v_k^l(1,\om_k(1;\xi,\om_j(\xi;x,t),\eps)) d \xi
\nonumber\\
&+&
 \tilde  e_{jk}(\xi,x,t,\eps)
 v_k^l(1,\om_k(1;\xi,\om_j(\xi;x,t),\eps))  \Big|_{\xi=x_{jk}(x,t,\eps)}^{x},
\end{eqnarray}
where
\begin{eqnarray*}
&e_{jk}(\xi,x,t,\eps)= d_j(\xi,x,t)
c_k(x_k(\xi,\om_j(\xi;x,t),\eps),\xi,\om_j(\xi;x,t),\eps)\d_\eps x_k(\xi,\om_j(\xi;x,t),\eps)&\nonumber\\
&\times
a_k(x_k(\xi,\om_j(\xi;x,t),\eps),s)\d_3\om_k(x_k(\xi,\om_j(\xi;x,t),\eps);\xi,\om_j(\xi;x,t),\eps)
\left(a_ja_k^\eps\be_{jk}^\eps\right)(\xi,\om_j(\xi;x,t)),&\nonumber\\[2mm]
 &\tilde  e_{jk}(\xi,x,t,\eps)= d_j(\xi,x,t)
c_k(1,\xi,\om_j(\xi;x,t),\eps)\d_\eps\om_k(1;\xi,\om_j(\xi;x,t),\eps)&
\nonumber\\ &\times
\left(a_ja_k^\eps\be_{jk}\right)(\xi,\om_j(\xi;x,t))\d_3\om_k(1;\xi,\om_j(\xi;x,t),\eps).&
\end{eqnarray*}
To finish with the first summand in \reff{rep_(u-v)_smooth} it remains, similarly to \reff{by_parts},
apply the conditions   \reff{eq:L3}, \reff{eq:L5},  \reff{eq:L4}, \reff{eq:lim_0},   and
the formulas \reff{dx}--\reff{d_eps}.

The last two summands in \reff{rep_(u-v)_smooth} are
treated by means of  the assumptions \reff{eq:L2}, \reff{eq:L3}, \reff{eq:L4} (entailing,
in particular, the uniform boundedness of the operators $B$ and $F$ restricted
to $C\left(\overline\Pi_s\setminus\Pi_{s+2d}\right)$) as well as by the smoothing apriori estimate
\beq\label{apr_v}
\|v\|_{C\left(\overline\Pi_{s+d}\setminus\Pi_{s+2d}\right)^n}+
\|\d_xv\|_{C\left(\overline\Pi_{s+d}\setminus\Pi_{s+2d}\right)^n}\le C\|\vphi\|_{C\left([0,1]\right)^n},
\ee
where the constant $C>0$ depends on $d$ but not on $\eps\le\eps_0$ and $s\in\R$.
We are therefore reduced to  prove the estimate \reff{apr_v}. To this end, we start with the operator
representation of $v$ in $\overline\Pi_{s+d}\setminus\Pi_{s+2d}$, namely,
$$
v=B^\eps Rv+D^\eps v.
$$
After a number of  iterations   we derive the following formula suitable for our purposes:
\beq\label{v}
v=\sum_{i=0}^{k-1}(B^\eps R)^i\left(D^\eps B^\eps R +\left(D^\eps\right)^2\right)v.
\ee
The estimate \reff{apr_v} now readily follows from the smoothing property  in $x$
 of the operators $D^\eps B^\eps$ and $\left(D^\eps\right)^2$ and the apriori estimate
\reff{apriori} with $2d$ in place of $T$. Showing the smoothing property  of the operators
$D^\eps B^\eps$ and $\left(D^\eps\right)^2$ in $x$,
we follow a similar argument as in the proof above of the smoothing property in $\eps$.
We illustrate this by example of the operator $\left(D^\eps\right)^2$ (and similarly for $D^\eps B^\eps$):
We take into account that $\left[\left(D^\eps\right)^2v^l\right]_j(x,t)$ on $\overline\Pi_{s+d}\setminus\Pi_{s+2d}$
is given by the formula \reff{D11} where  $b_{jk}$ is replaced by $b_{jk}^\eps$; $x_j(x,t)\equiv 0$ if $j\le m$; and
 $x_j(x,t)\equiv 1$ if $j>m$. Below we therefore drop the dependence of $x_j$ on $x$ and $t$.
Changing the order of integration, we have
\begin{eqnarray}
\lefteqn{
\d_x\left[\left(\left(D^\eps\right)^2v^l\right)_j(x,t)\right]
}\nonumber\\\nonumber
&=&\sum_{k=1\atop k\not=j}^n\sum_{i=1\atop i\not=k}^n
\int_{x_j}^x \int_\eta^x \d_x\bigl[d_{jki}(\xi,\eta,x,t,\eps)b_{jk}(\xi,\om_j(\xi;x,t,\eps),\eps)\bigr]
v_i^l(\eta,\om_k(\eta;\xi,\om_j(\xi;x,t,\eps),\eps)) d \xi d \eta\nonumber\\
&+&\sum_{k=1\atop k\not=j}^n\sum_{i=1\atop i\not=k}^n
\int_{x_j}^x \int_\eta^x {d}_{jki}(\xi,\eta,x,t,\eps)b_{jk}(\xi,\om_j(\xi;x,t,\eps),\eps)\nonumber\\
&\times&\d_3\om_k(\eta;\xi,\om_j(\xi;x,t,\eps),\eps)\d_x\om_j(\xi;x,t,\eps)\d_2
v_i^l(\eta,\om_k(\eta;\xi,\om_j(\xi;x,t,\eps),\eps)) d \xi d \eta.\label{final}
\end{eqnarray}
Let us transform the second summand similarly to \reff{by_parts}: For given  $k\not=j$ and $i\not=k$ we have
(using the assumptions \reff{eq:dif} and \reff{cass})
\begin{eqnarray*}
\lefteqn{\int_{x_j}^x \int_\eta^x {d}_{jki}(\xi,\eta,x,t,\eps)b_{jk}(\xi,\om_j(\xi;x,t,\eps),\eps)}\\
&\times&\d_3\om_k(\eta;\xi,\om_j(\xi;x,t,\eps),\eps)\d_x\om_j(\xi;x,t,\eps)
\d_2v_i^l(\eta,\om_k(\eta;\xi,\om_j(\xi;x,t,\eps),\eps)) d \xi d \eta\\
&=&\int_{x_j}^x \int_\eta^x {d}_{jki}(\xi,\eta,x,t,\eps)\d_3\om_k(\eta;\xi,\om_j(\xi;x,t,\eps),\eps)\d_x\om_j(\xi;x,t,\eps)\\
&\times&b_{jk}(\xi,\om_j(\xi;x,t,\eps),\eps)\Bigl[\bigl(\d_\xi\om_k\bigr)(\eta;\xi,\om_j(\xi;x,t,\eps),\eps)\Bigr]^{-1}
\bigl(\d_\xi v_i^l\bigr)
(\eta,\om_k(\eta;\xi,\om_j(\xi;x,t,\eps),\eps)) d \xi d \eta\\
&=&\int_{x_j}^x \int_\eta^x {d}_{jki}(\xi,\eta,x,t,\eps)\d_x\om_j(\xi;x,t,\eps)\bigl(a_ka_j\ga_{jk}\bigr)
(\xi,\om_j(\xi;x,t,\eps),\eps)\\
&\times&
\bigl(\d_\xi v_i^l\bigr)(\eta,\om_k(\eta;\xi,\om_j(\xi;x,t,\eps),\eps)) d \xi d \eta\\
&=&\int_{x_j}^x \int_\eta^x \tilde{d}_{jki}(\xi,\eta,x,t,\eps)\bigl(\d_\xi v_i^l\bigr)
(\eta,\om_k(\eta;\xi,\om_j(\xi;x,t,\eps),\eps))
d \xi d \eta\\
\nonumber&=&-\int_{x_j}^x\int_\eta^x \d_\xi \tilde{d}_{jki}(\xi,\eta,x,t,\eps)
v_i^l\left(\eta,\om_k(\eta;\xi,\om_j(\xi;x,t,\eps),\eps)\right)
d\xi d\eta\\
&+&\int_{x_j}^x\Big[\tilde{d}_{jki}(\xi,\eta,x,t,\eps)v_i^l\left(\eta,\om_k(\eta;\xi,\om_j(\xi;x,t,\eps),\eps)
\right)\Big]_{\xi=\eta}^{\xi=x}
d\eta,
\end{eqnarray*}
where
$$
\tilde{d}_{jki}(\xi,\eta,x,t) = {d}_{jki}(\xi,\eta,x,t,\eps)\d_x\om_j(\xi;x,t,\eps)
\left(
a_ka_j\ga_{jk}
\right)
(\xi,\om_j(\xi;x,t,\eps),\eps).
$$
Now in \reff{final} we can pass to the limit  as $l\to\infty$ and then to the right-hand side
apply the apriori estimate \reff{apriori}. Combining the resulting inequality with the formula \reff{v}
and the apriori estimate \reff{apriori}
gives \reff{apr_v}. The proof of  Theorem \ref{robust} is therewith complete.

\section*{Acknowledgments}
The first author  was supported by the Alexander von Humboldt Foundation.
Irina Kmit and Lutz Recke acknowledge support of the DFG Research Center {\sc Matheon}
mathematics for key technologies (project D8).

\end{document}